\documentclass[a4paper,10pt]{article}
\usepackage[utf8x]{inputenc}
\usepackage{amsmath,amssymb, amsthm}
\usepackage{enumerate}
\usepackage{verbatim}

\newcommand{\N}{\mathbb{N}}
\newcommand{\Z}{\mathbb{Z}}

\newcommand{\R}{\mathbb{R}}
\newcommand{\C}{\mathbb{C}}

\newcommand{\OreGen}{R[x;\sigma,\delta]}

\newcommand{\DiffPol}{R[x;\identity_R,\delta]}

\DeclareMathOperator{\identity}{id}
\DeclareMathOperator{\Cen}{Cen}

\makeatletter
\def\imod#1{\allowbreak\mkern10mu({\operator@font mod}\,\,#1)}
\makeatother

\theoremstyle{plain}
\newtheorem{theorem}{Theorem}[section]
\newtheorem{lemma}[theorem]{Lemma}
\newtheorem{remark}[theorem]{Remark}
\newtheorem{proposition}[theorem]{Proposition}
\newtheorem{corollary}[theorem]{Corollary}
\newtheorem{definition}[theorem]{Definition}

\title{Centralizers and pseudo-degree functions }
\author{Johan Richter % 
\footnote{M\"{a}lardalen University, Academy of Education, Culture and Communication, Box 883, 721 23 V\"{a}ster\aa s, Sweden E-mail: johan.richter@mdh.se }
}

\begin{document}

\maketitle

\begin{abstract}

This paper generalizes a proof of certain results by Hellstr\"{o}m and Silvestrov \cite{ergodipotent} on centralizers in graded algebras. We study centralizers in certain algebras with valuations. We prove that the centralizer of an element in these algebras is a free module over a certain ring. Under further assumptions we obtain that the centralizer is also commutative.  \\


\end{abstract}

\section{Introduction}

The British mathematicians Burchnall and Chaundy studied, in a series of papers in the 1920s and 30s \cite{BurchnallChaundy1,BurchnallChaundy2, BurchnallChaundy3}, the properties of commuting pairs of ordinary differential operators. The following theorem is essentially found in their papers. 
\begin{theorem}\label{thm_BC}
 Let $P=\sum_{i=0}^n p_i D^i$ and $Q= \sum_{j=0}^{m} q_j D^j$ be two commuting elements of $T$ with constant leading coefficients. Then there is a non-zero polynomial $f(s,t)$ in two commuting variables over $\C$ such that 
$f(P,Q) =0$. Note that the fact that $P$ and $Q$ commute 
guarantees that $f(P,Q)$ is well-defined.
\end{theorem} 

The result of Burchnall and Chaundy was rediscovered independently during the 70s by researchers in the area of PDEs. It turns out that several important equations can be equivalently 
formulated as a condition that a pair of differential operators commute. These differential equations are completely integrable as a result, which roughly means that they possess an infinite number of conservation
laws. In fact  Theorem~\ref{thm_BC} was rediscovered by Kricherver \cite{Krichever} as part of his research into integrable systems. 

To state some generalizations of Burchnall's and Chaundy's result  we shall recall a definition. 

\begin{definition}

Let $R$ be a ring, $\sigma$ an endomorphism of $R$ and $\delta$ an additive function, $R \to R$, satisfying 
\begin{equation*}
\delta(ab) = \sigma(a)\delta(b)+\delta(a)b
\end{equation*}
for all $a,b \in R$. (Such $\delta$:s are known as $\sigma$-derivations.) The \emph{Ore extension} $\OreGen$ is the polynomial ring $R[x]$ equipped with a new multiplication such that $xr=\sigma(r)x+\delta(r)$ for all $r \in R$. Every element of $\OreGen$ can be written uniquely as  $\sum_i a_i x^i$ for some $a_i \in R$.

If $\sigma=\identity$ then $\DiffPol$ is called a \emph{differential operator ring}. If $P= \sum_{i=0}^n a_i x^i$, with $a_n \neq 0$, we say that $P$ has \emph{degree} $n$. We say that the zero element has degree $-\infty$. 
\end{definition}

The ring of differential operators studied by Burchnall and Chaundy can be taken to be the Ore extension $T=C^{\infty}(\R,\C)[D;\identity,\delta]$, where $\delta$ is the ordinary derivation.

In a paper by Amitsur \cite{Amitsur} one can find the following theorem.

\begin{theorem}\label{thm_DiffField}
 Let $K$ be a field of characteristic zero with a derivation $\delta$. Let $F$ denote the subfield of constants. (By a \emph{constant} we mean an element that is mapped to zero by the derivation.) Form the differential operator ring $S=K[x; \identity,\delta]$, and let $P$ be an element of $S$ of 
degree $n>0$. Set $F[P]= \{ \sum_{j=0}^m b_j P^j \ | \ b_j \in F \ \}$, the ring of polynomials in $P$ with constant coefficients. 
Then the centralizer of $P$ is a commutative subring of
$S$ and a free $F[P]$-module of rank at most $n$.  
\end{theorem}

Later authors have found other contexts where Amitsur's method of proof can be made to work. We mention an article by Goodearl and Carlson \cite{GoodearlCarlson}, and one by Goodearl alone \cite{GoodearlPseudo}, that generalize Amitsur's result to a wider class of rings. The proof has also been generalized by Bavula \cite{Bavula}, Mazorchouk \cite{Mazorchuk} and Tang \cite{Xin}, among other authors. As a corollary of these results, one can recover Theorem \ref{thm_BC}.

This paper is most directly inspired by a paper by Hellstr\"{o}m and Silvestrov \cite{ergodipotent}, however. Hellstr\"{o}m and Silvestrov study graded algebras satisfying a condition they call $l$-BDHC (short for ``Bounded-Dimension Homogeneous Centralizers''). 

\begin{definition}
Let $K$ be a field, $\ell$ a positive integer and $S$ a $\Z$-graded $K$-algebra. The homogeneous components of the gradation are denoted $S_m$, for $m \in \Z$. Let $\Cen(n,a)$, for $n \in \Z$ and $a \in S$, denote the elements in $S_n$ that commute with $a$. We say that $S$ has $\ell$-BDHC if for all $n \in \Z$, nonzero $m \in \Z$ and nonzero $a \in S_m$, it holds that $\dim_K \Cen(n,a) \leq \ell$.  

\end{definition}

Hellstr\"{o}m and Silvestrov apply the ideas of Amitsur's proof. They need to modify them however, especially to handle the case when $\ell>1$. 

To explain their results further, we introduce some more of their notation. Denote by $\pi_n$ the projection, defined in the obvious way, from $S$ to $S_n$. Hellstr\"{o}m and Silvestrov define a function $\bar{\chi}: A\setminus \{ 0\} \to \Z$ by
\begin{equation*}
\bar{\chi}(a) = \max\{ \, n \, \in \Z \, | \, \pi_n(a) \neq 0 \, \},
\end{equation*}
and set $\bar{\chi}(0) = -\infty$. %(They also define an analogue with $\min$ instead of $\max$.) 
Set further $\bar{\pi}(a) = \pi_{\bar{\chi}(a)}(a)$. 

Now we have introduced enough notation to state the relevant results. The following result is the main part of Lemma 2.4 in their paper. 

\begin{theorem}\label{thm_HS1}
Assume $S$ is a $K$-algebra with $l$-BDHC and that there are no zero divisors in $S$. If $a \in S\setminus S_0$ is such that $\bar{\chi}(a)=m >0$ and $\bar{\pi}(a)$ is not invertible in $S$, then there exists a finite $K[a]$-module basis $\{b_1,\ldots, b_k\}$ for the centralizer of $a$. Furthermore $k \leq ml$. 
\end{theorem}

The reason they refer to it as a lemma is that their main interest is in the following corollary of this result. (Which is proved the same way as Corollary \ref{corAlgDep} in this paper.) 

\begin{theorem}\label{thm_HS2}
Let $K$ be a field and assume the $K$-algebra $S$ has $l$-BDHC and that there are no zero divisors in $S$. If $a \in S\setminus S_0$ and $b\in S$ are such that $ab = ba$, $\bar{\chi}(a)>0$ and $\bar{\pi}(a)$ is not invertible in $S$, then there exists a nonzero polynomial $P$ in two commuting variables with coefficients from $K$ such that $P(a,b)=0$. 
\end{theorem}

Theorem \ref{thm_HS2} is directly analogous to Theorem \ref{thm_BC}.

Hellstr\"om and Silvestrov also have a result asserting that certain centralizers are commutative. Their proof can be made to work in the case when $A$ has $1$-BDHC. 

\begin{theorem}\label{thm_HS3}
Assume the $K$-algebra $S$ has $1$-BDHC and that there are no zero divisors in $S$. If $a \in S\setminus S_0$ satisfies $\bar{\chi}(a)=m >0$ and $\bar{\pi}(a)$ is not invertible in $S$, then there exists a finite $K[a]$-module basis $\{b_1,\ldots, b_k\}$ for the centralizer of $a$. The cardinality, $k$, of the basis divides $m$. Furthermore the centralizer of $a$ is commutative. 
\end{theorem}

It shall be the goal of this paper to generalize the results we have cited from \cite{ergodipotent}. 

\subsection{Notation and conventions}

$\Z$ will denote the integers. 

If $R$ is a ring then $R[x_1,x_2, \ldots x_n]$ denotes the ring of polynomials over $R$ in central indeterminates $x_1,x_2,\ldots,x_n$. 

All rings and algebras are assumed to be associative and unital.

Let $R$ be a commutative ring and $S$ an $R$-algebra. Two commuting elements, $p,q \in S$, are said to be \emph{algebraically dependent} (over $R$) if there is a non-zero polynomial, $f(s,t) \in R[s,t]$, such that $f(p,q)=0$, in which 
case $f$ is called an annihilating polynomial. 

If $S$ is a ring and $a$ is an element in $S$, the \emph{centralizer} of $a$, denoted $C_S(a)$, is the set of all elements in $S$ that commute with $a$. 

By $K$ we will always denote a field.

\section{Centralizers in algebras with degree functions}

Upon reading the proofs in \cite{ergodipotent} closely it turns out that they are based upon certain properties of the function $\bar{\chi}$ they define.   We shall axiomatize the properties that are needed to make their proof work. 

\begin{definition}
Let $K$ be a field and let $S$ be a $K$-algebra. A function, $\chi$, from $S$ to $\Z \cup \{-\infty\}$ is called a pseudo-degree function if it satisfies the following conditions: 

\begin{itemize}
\item $\chi(a) = -\infty$ iff $a=0$,
\item $\chi(ab) = \chi(a)+\chi(b)$ for all $a,b \in S$, 
\item $\chi(a+b) \leq \max(\chi(a),\chi(b))$,
%\item $\chi(a+b) = \chi(a)$ if $\chi(b)<\chi(a)$. 
\end{itemize}

\end{definition}

This is essentially a special case of the concept of a \emph{valuation}.

We also need a condition that can replace $l$-BDHC. We formulate it next.

\begin{definition}

Let $K$ be a field and $S$ a $K$-algebra with a pseudo-degree function, $\chi$, and let $\ell$ be a positive integer. A subalgebra ,$B \subset A$, is said to satisfy condition $D(\ell)$ if  $\chi(b) \geq 0$ for all non-zero $b \in B$ and if, whenever we have $\ell+1$ elements $b_1, \ldots, b_{l+1} \in B$, all mapped to the same integer by $\chi$, there exist $\alpha_1,\ldots, \alpha_{\ell+1} \in K$, not all zero, such that $\chi\left(\sum_{i=1}^{l+1} \alpha_i b_i\right) < \chi(b_1)$. 

\end{definition}

\begin{remark}
Note that the requirement that $\alpha_1,\ldots,\alpha_{l+1}$ are mapped to the same integer by $\chi$ excludes the possibility that they are equal to $0$.   

\end{remark}

\begin{remark}\label{remScalarDeg}

Suppose that $S$ is a $K$-algebra and $a \in S$ is such that $C_S(a)$ satisfies condition $D(\ell)$ for some $\ell$. If $b$ is an invertible element then $\chi(b^{-1}) = -\chi(b)$. So all invertible elements of $C_S(a)$ must be mapped to zero by $\chi$. In particular the non-zero scalars are all mapped to zero by $\chi$. 

\end{remark}

\begin{lemma}\label{lemAdd}

Suppose that $S$ is an $K$-algebra and $\chi$ is a pseudo-degree function on $S$ that maps all the non-zero scalars to zero. Then if $a,b \in S$  are such that $\chi(b)<\chi(a)$, the identity 
\begin{equation}
\chi(a+b) =\chi(a)
\end{equation}
holds.
\end{lemma}

\begin{proof}
On the one hand we find $\chi(a+b) \leq \max(\chi(a),\chi(b))=\chi(a)$. On the other hand 
$\chi(a) = \chi(a+b-b) \leq \max(\chi(a+b),\chi(b))$ Since $\chi(b)<\chi(a)$ we must have 
$\chi(a) \leq \chi(a+b)$. 
\end{proof}

We now proceed to prove an analogue of Theorem \ref{thm_HS1}, using just the existence of some pseudo-degree function and the condition $D(\ell)$.

\begin{theorem}\label{thm_BoundDim}
Let $K$ be a field and let $S$ be a $K$-algebra. Suppose $S$ has a pseudo-degree function, $\chi$.

Let $a$ be an element of $S$, with $m=\chi(a)>0$, such that $C_S(a)$ satisfies condition $D(\ell)$ for some positive integer $\ell$. Then $C_S(a)$ is a free $K[a]$-module of rank at most $\ell m$.

\end{theorem}

\begin{proof}

Construct a sequence $b_1,b_,\ldots $ by setting $b_1=1$ and choosing $b_{k+1} \in C_S(a)$ such that $\chi(b_{k+1})$ is minimal subject to the restriction that $b_{k+1}$ does not lie in the $K[a]$-linear span of $\{b_1, \ldots, b_k\}$. We will show later in the proof that such a sequence has at most $lm$ elements. 

We first claim that 
\begin{equation}\label{degreeSum}
\chi \left(\sum_{i=1}^k \phi_i b_i\right) = \max_{i\leq k} (\chi(\phi_i) +\chi(b_i)),
\end{equation} 
for any $\phi_1,\ldots \phi_k \in K[a]$. We show this by induction on $n=\max_{i\leq k} (\chi(\phi_i) +\chi(b_i))$. It is clear that the left-hand side of \eqref{degreeSum} is never greater than the right-hand side. When $n=-\infty$ Equation \eqref{degreeSum} holds since in that case all $\phi_i=0$. If $n=0$, Equation \eqref{degreeSum} holds since $\chi(b)\geq 0$ for all non-zero $b \in C_S(a)$. That $\chi(b) \geq 0$ for all non-zero $b$ in $C_S(a)$ also means that no value of $n$ between $-\infty$ and $0$ is possible. 

For the induction step, assume \eqref{degreeSum} holds when the right-hand side is strictly less than $n$. To verify that it holds for $n$ as well, we can assume without loss of generality  that $\chi(\phi_k) + \chi(b_k)=n$, since if $\chi(\phi_j b_j)<n$ for some term $\phi_j b_j$ we can drop it without affecting either side of \eqref{degreeSum}, by Lemma \ref{lemAdd}. If $\phi_k \in K$ then $\chi(\phi_k)=0$, by Remark \ref{remScalarDeg}, and thus $\chi(b_k)=n$. By the choice of $b_k$ it then follows that $\chi(\sum_{i=1}^k \phi_i b_i) \geq n$, as otherwise $\sum_{i=1}^k \phi_i b_i$ would have been picked instead of $b_k$. If $\phi_k \notin K$, then $\chi(b_k) <n$ and thus $\chi(b_i) < n$ for $i=1,\ldots k$. Let $r_1,\ldots,r_k \in K$ and $\xi_1,\ldots,\xi_k \in K[a]$ be such that 
$\phi_i = a\xi_i +r_i$ for $i=1,\ldots,k$. We have $\chi(\sum_{i=1}^k r_i b_i) <n$ and thus by Lemma \ref{lemAdd} and the assumptions on $\chi$ we get
\begin{equation*}
\chi\left(\sum_{i=1}^k \phi_i b_i \right) = \chi\left(\sum_{i=1}^k a\xi_i b_i +\sum_{i=1}^k r_ib_i \right) = \chi\left(a \sum_{i=1}^k \xi_i b_i \right) = m+\chi\left(\sum_{i=1}^k \xi_i b_i \right).
\end{equation*} 
We also have that $\max_{i\leq k} (\chi(\phi_i) +\chi(b_i)) = m +\max_{i\leq k} (\chi(\xi_i) +\chi(b_i))$.  By the induction hypothesis 
\begin{equation*}
\chi\left(\sum_{i=1}^k \xi_i b_i \right) = \max_{i\leq k} (\chi(\xi_i) +\chi(b_i)),
\end{equation*}
which completes the induction step. 

We now show that if $\chi(b_i) = \chi(b_j)$ for some $i\leq j$ then $j-i<l.$ Suppose $b_{1} ,\ldots,b_{l+1} $ all are mapped to zero by $\chi$. Then there exists $\alpha_1,\ldots,\alpha_{l+1}$, not all zero, such that 
\begin{equation*}
\chi\left(\sum_{i=1}^{l+1} \alpha_i b_i \right) <0,
\end{equation*}
which is impossible since $ \sum_{i=1}^{l+1} \alpha_i b_i  \in C_S(a)$. 

Suppose now instead that $b_j,\ldots,b_{j+l}$ are all mapped to the same positive integer, $q$, by $\chi$. Then there exists $\alpha_j,\ldots,\alpha_{j+l} \in K$, not all zero, such that 
\begin{equation*}
\chi\left( \sum_{i=j}^{j+l} \alpha_i b_i \right) < q.
\end{equation*}

But this contradicts \eqref{degreeSum}

%But this is only possible if some $b_i$ lies in the $K[a]$-linear span of the previous %terms of the sequence, which contradicts the way it was chosen. 

It remains only to show that the sequence $(b_i)$ contains only $lm$ elements. We will prove that every residue class $(\mod m$) can only contain at most $l$ elements. Suppose to the contrary, that we had elements $c_1,\ldots, c_{l+1}$, belonging to the sequence $(b_i)$ and all satisfying that $ \chi(c_i) = n \imod{m}$. Set $k=\max_{1\leq i \leq l+1} (\chi(c_i))$ and define $\gamma_i = a^{\frac{k-\chi(c_i)}{m}}.$ Then $\chi(\gamma_i c_i) = k$, for all $i \in \{1, \ldots, l+1\}$, which implies that there exists  $\alpha_1,\ldots,\alpha_{l+1} \in K$, such that 
\begin{equation*}
\chi\left( \sum_{i=j}^{j+l} \alpha_i \gamma_i c_i \right) < k.
\end{equation*}
But this once again contradicts \eqref{degreeSum}.  
 
\end{proof}

We can also prove a result on the algebraic dependence of pairs of commuting elements.

\begin{corollary}\label{corAlgDep}
Let $S$ be a $K$-algebra with a pseudo-degree function, $\chi$. Let $a \in S$ be such that $C_S(a)$ satisfies Condition $D(l)$ for some $l>0$. Let $b$ be any element in $C_S(a)$. Then there exists a nonzero polynomial $P(s,t)\in K[s,t]$ such that $K(a,b) =0$. (Note that $K(a,b)$ is well-defined when $a,b$ commute.)
\end{corollary}

\begin{proof}

Since $C_S(a)$ has finite rank as a $K[a]$-module the elements $b,b^2,\ldots$ can not all be linearly independent over $K[a]$. Thus there exists $f_1(x),\ldots, f_k(x) \in K[x]$ , not all zero,  such that $\sum_{i=0}^k f_i(a) b^i =0$. Then $P(s,t) = \sum_{i=0}^k f_i(s) t^i =0$ is a polynomial with the desired property. 

\end{proof}

We can also prove a result asserting that certain centralizers are commutative, though for that we need to assume that $C_S(a)$ satisfies condition $D(1)$. 

\begin{theorem}\label{thmComCen}
Let $K$ be a field and suppose $S$ is a $K$-algebra. Let $S$ have a pseudo-degree function, $\chi$. If $a\in S$ satisfies $\chi(a)=m>0$ and $C_S(a)$ satisfies condition $D(1)$ then: 
\begin{enumerate}
\item $C_S(a)$ has a finite basis as a $K[a]$-module, the cardinality of which divides $m$.
\item $C_S(a)$ is a commutative algebra. 
\end{enumerate}
\end{theorem}

\begin{proof}

By Theorem \ref{thm_BoundDim} it is clear that there is a subset $H$ of $\{1, \ldots, m\}$
and elements $(b_i)_{i \in H}$ such that the $b_i$ form a basis for $C_S(a)$. By the proof of Theorem \ref{thm_BoundDim} it is also clear that $\chi(b_i) \neq \chi(b_j)$ if $i \neq j$. Without loss of generality we can assume $\chi(b_i)=i$ for all $i  \in H$. We can map $H$ into $\Z_m$ in a natural way. Denote the image by $G$. We want to show $G$ is a subgroup, for which it is enough to show that it is closed under addition. 

Suppose $g,h \in G$. There exists $i,j \in H$, with $i \equiv g \imod{m}$ and $j \equiv h \imod{m}$. We can write $b_ib_j= \sum_{k\in H} \phi_k b_k$, for some $\{b_k \}$. It follows that 
\begin{equation*}
g+h \equiv i+j = \chi(b_ib_j)  = \max (\chi(\phi_k)+\chi(b_k)) \equiv \chi(b_k ) =k \imod{m} 
\end{equation*}
 for some $k \in H$. 

Since $G$ is a subgroup of $\Z_m$ it is clear that the cardinality of $G$, which is also the cardinality of $H$, must divide $m$. 

$G$ is cyclic. Let $g$ be a generator of $G$. Consider the algebra generated by $b_i$ and $a$, where $i \equiv g \imod{m}$. It is a commutative algebra and a sub-$K$-vector space of $C_S(a)$. Denote it by $E$. If $c$ is any element of $C_S(a)$ we can write $c=e+f$, where $e \in E$ and $\chi(f) <mi$, since if $\chi(c)\geq mi$ then there exists $k \leq m$ and $j \in \N$ such that $\chi(a^j b_i^k)=\chi(c)$ and thus there exists $\alpha \in K$ such that $\chi(c-\alpha a^j b_i^k) < \chi(c)$. 

Thus the quotient $C_S(a)/E$ is finite-dimensional. Each $f\in K[a]$ gives rise to an endomorphism on $C_S(a)/E$, by the action of multiplication by $f$. Since $K[a]$ is infinite-dimensional and the endomorphism ring of $C_S(a)/E$ is finite-dimensional, there is some nonzero $\phi \in K[a]$ that induces the zero endomorphism. But this means that $\phi c \in E$ for any $c \in C_S(a)$. 

Now let $c_1,c_2$ be two arbitrary elements of $C_S(a)$. Since $E$ is commutative, and everything in $C_S(a)$ commutes with $\phi$, it follows that 
\begin{equation*}
\phi^2 c_1 c_2 = \phi c_1 \cdot \phi c_2=  \phi c_2 \cdot \phi c_1 = \phi^2 c_2 c_1.
\end{equation*} 
Since $C_S(a)$ is a domain it follows that $c_1c_2=c_2c_1$ and thus that $C_S(a)$ is commutative. 
\end{proof}

\section{Examples}

Theorems \ref{thm_HS1},\ref{thm_HS2} and \ref{thm_HS3} follow from our results combined with Lemma 2.2 and Lemma 2.4 in \cite{ergodipotent}. But our results can also be applied in certain situations that are not covered by the results in \cite{ergodipotent}. 

\begin{proposition}
Let $K$ be a field. Set $R=K[y]$, let $\sigma$ be an endomorphism of $R$ such that $s=\deg_y(\sigma(y))>1$ and let $\delta$ be a $\sigma$-derivation. Form the Ore extension $S=\OreGen$. If $a \in S \setminus K$ then $C_S(a)$ is a free $K[a]$-module of finite rank and a commutative subalgebra of $S$.  

\end{proposition} 

\begin{proof}
If $a \in K[y] \setminus K$ then $C_S(a) =K[y]$ and the claim is true. So suppose that $a \notin K[y]$. We shall apply Theorem \ref{thmComCen}. To do so we need a pseudo-degree function. 

The notion of the degree of an element in $S$ with respect to $x$ was defined in the introduction of this article. Denote the degree of an element $b$ by $\chi(b)$. It is easy to see that $\chi$ satisfies all the requirement to be a pseudo-degree function. We proceed to show that $C_S(a)$ satisfies condition $D(1)$. Certainly it is true that $\chi(b)\geq 0$ for all nonzero $b \in C_S(a)$. 

Let $b$ be a nonzero element of $S$ that commutes with $a$, such that $\chi(b) =n$. Suppose $\chi(a) =m$. By equating the highest order coefficient of $ab$ and $ba$ we find that 
\begin{equation}\label{eqHigCoe}
a_m \sigma^m (b_n) = b_n \sigma^n (a_m), 
\end{equation}
where $a_m$ and $b_n$ denote the highest order coefficients of $a$ and $b$, respectively. (Recall that these are polynomials in $y$.) We equate the degree in $y$ of both sides of \eqref{eqHigCoe} and find that 
\begin{equation*}
\deg_y(a_m) +s^m \deg_y(b_n) = \deg_y(b_n) +s^n \deg_y (a_m),
\end{equation*}
which determines the degree of $b_m$ uniquely. It follows that the solutions of \eqref{eqHigCoe} form a $K$-sub space of $K[y]$ that is at most one-dimensional. This in turn implies that condition $D(1)$ is fulfilled. 

We have now verified all the hypothesis necessary to apply Theorem \ref{thmComCen}. 
\end{proof}

\section*{Acknowledgements}

This work was performed, in part, while the author was employed at Lund University and in part while the author was employed at M\"{a}lardalen University. 

The author wishes to thank Lars Hellstr\"{o}m and Johan \"{O}inert for helpful discussions.

\bibliographystyle{amsplain}

\bibliography{General}

\end{document}